\documentclass[12pt]{amsart}
\usepackage{amssymb,amsmath,amsthm,fullpage,color,soul, mathrsfs}
\usepackage{pdfsync}
\usepackage{enumerate}

\newtheorem{thm}{Theorem}[section]

\theoremstyle{definition}

\theoremstyle{remark}
\newtheorem{rem}[thm]{Remark}

\numberwithin{equation}{section}

\DeclareMathOperator{\Span}{span}

\newcommand{\Om}{\Omega}

\newcommand{\R}{\mathbb R}

\newcommand{\N}{\mathbb N}
\newcommand{\C}{\mathbb C}

\DeclareMathOperator{\Imm}{Im}

\DeclareMathOperator{\Rre}{Re}
\DeclareMathOperator{\Dom}{Dom}

\newcommand{\bd}{\textrm{b}}

\newcommand{\p}{\partial}

\newcommand{\z}{\bar z}
\newcommand{\w}{\bar w}
\newcommand{\dbar}{\bar\partial}
\newcommand{\dbars}{\bar\partial^*}

\newcommand{\vp}{\varphi}

\newcommand{\atopp}[2]{\genfrac{}{}{0pt}{2}{#1}{#2}}

\newcommand{\ep}{\epsilon}
\newcommand{\I}{\mathcal{I}}

\newcommand{\la}{\langle}
\newcommand{\ra}{\rangle}

\newcommand{\opK}{\mathcal K}
\newcommand{\opD}{\mathcal D}

\begin{document}

\title[The Bergman Kernel on Forms]{The Bergman Kernel on Forms: General Theory}

\author{Andrew Raich}

\thanks{The author was partially supported by NSF grant DMS-1405100. I would also like to thank Phil Harrington for several helpful discussions on this project.}

\address{Department of Mathematical Sciences, SCEN 309, 1 University of Arkansas, Fayetteville, AR 72701}
\email{araich@uark.edu}



\begin{abstract} The goal of this note is to explore the Bergman projection on forms. In particular, we show that some of most basic facts used to construct the Bergman kernel on functions, such as pointwise evaluation 
in $L^2_{0,q}(\Om)\cap\ker\dbar_q$,  fail for $(0,q)$-forms, $q \geq 1$. 
We do, however,
provide a careful construction of the Bergman kernel and explicitly compute the Bergman kernel on $(0,n-1)$-forms. 
In the ball in $\mathbb{C}^2$, we also show that the size of the Bergman kernel on $(0,1)$-forms is not governed by the control metric, in stark
contrast to Bergman kernel on functions.
\end{abstract}

\maketitle

%
%
\section{Introduction}\label{sec:intro}

On a domain $\Om\subset\C^n$, the Bergman projection $B_q$ is the
the orthogonal projection $B_q:L^2_{0,q}(\Om) \to \ker \dbar_q \cap L^2_{0,q}(\Om)$. The basic theory of the classical Bergman projection $B_0$ is, well, classical and can be found in any several complex variables
textbook, e.g., \cite{Kra01}. The Bergman projection $B_0$ is one of the most basic objects in the analysis of both one and several variables, and its mapping properties have been exhaustively (though not conclusively)
researched, as have formulas for its kernel. See, for example, \cite{Cat83,Cat87,KoNi65,FoKo72, PhSt77,   McN89, NaRoStWa89,  
ChDu06, NaSt06,McN94, McSt94, KhRa18, Fef74,DAn78, DAn94} for just a small samplings of the results in the literature. 
Surprisingly, when $q\geq 1$, only mapping properties have been investigated -- regularity properties for Bergman projects often follows from estimates of the $\dbar$-Neumann operator and Kohn's formula (see, for example,
\cite{HaRa15,BoSt90}).
There is essentially no literature about explicit construction of the kernels, pointwise size estimates, or geometry. 

A standard discussion of $B_0$ includes a formal construction of the integral kernel, its transformation law under biholomorphic mappings, and a computation of the Bergman kernel on
the ball (and perhaps the polydisk). One of the goals of this paper is to show that several of the main features of $B_0$ and its construction fail for $B_q$, $q\geq 1$. In particular, we show that:
\begin{enumerate}
\item Pointwise evaluation is not a bounded linear functional on $L^2_{0,q}(\Om)\cap \ker(\dbar_q)$;
\item It is unrealistic for a transformation formula to hold for $B_{p,q}(z,w)$ unless $p,q\in\{0,n\}$;
\item In $\C^2$, the Bergman kernel $B_1(z,w)$ on the ball does not behave according to the control geometry (in start constrast to $B_0(z,w)$).
\end{enumerate}

There is no additional information to be gained by looking at the Bergman projection on $L^2_{p,q}(\Om)$, so we focus on the $p=0$ case, except when we investigate the existence of transformation formulas
because the $B_{p,0}$ behaves worse that $B_0$.

We start by carefully constructing $B_q$, which, while
using well known Hilbert space and distribution theory, does not seem to appear in the literature. We then exploit Kohn's formula and the knowledge of the $\dbar$-Neumann problem in the top degree to give a general
formula for the Bergman projection $B_{n-1}$, and its associated integral kernel $B_{n-1}(z,w)$. We conclude the paper with a discussion on the ball. We compute $B_{n-1}$ explicitly and then restrict ourselves to the $\C^2$ case.
There, we observe  that the control geometry, which governs the size of $B_0(z,w)$, does not reflect the scaling present in the kernel $B_1(z,w)$.  We conclude with a remark about future directions.

Fix $q \geq 1$. The kernel,
$\ker\dbar_q$, is a closed subspace of $L^2_{0,q}(\Om)$, so the projection $B_q$ onto $\ker\dbar_q\cap L^2_{0,q}(\Om)$ can be given as a Fourier series in terms of a basis. 
The construction of $B_q$ can proceeds as follows:
suppose that
$\{\phi_j\}_{j=1}^\infty$ is an orthonormal basis of $\ker \dbar_q \cap L^2_{0,q}(\Om)$. The vector projection of $f \in L^2_{0,q}(\Om)$ onto $\Span \phi_j$ is $(f,\phi_j)\phi_j$ where the inner product
\[
(f,\phi_j) = \int_{\Om} \la f, \phi_j \ra\, dV = \int_\Om f \wedge \star \phi_j.
\]
where $\star$ is the Hodge-$\star$ operator (see, e.g., \cite[p.208]{ChSh01}) and $dV$ is Lebesgue measure. The orthogonal projection of $f$ on $\ker\dbar_q \cap L^2_{0,q}(\Om)$ is therefore given by the Fourier series
\[
B_q f(z) = \sum_{j=1}^\infty (f,\phi_j) \phi_j(z)
\]
where the sum converges in $L^2_{0,q}(\Om)$.

Working formally, we see that
\[
B_q f(z) = \sum_{j=1}^\infty \Big(\int_\Om f(w) \wedge \star \phi_j(w)\Big) \phi_j(z) = \int_\Om f(w) \wedge\Big(\sum_{j=1}^\infty \star\phi_j(w) \wedge \phi_j(z)\Big).
\]
This suggests that the Bergman kernel ought to be
\[
B_q(z,w) = \sum_{j=1}^\infty \star\phi_j(w) \wedge \phi_j(z)
\]
for \emph{any} orthonormal basis $\{\phi_j\}$ of $\ker \dbar_q \cap L^2_{0,q}(\Om)$. For this formula to be rigorous, of course, 
the sum defining $B_q(\cdot,w)$ must converge in $L^2_{0,q}(\Om)$, be independent of the orthonormal system $\{\phi_j\}$, and 
be the orthogonal projection onto $\ker\dbar_q\cap L^2_{0,q}(\Om)$. This is contain in Theorem \ref{thm:structure theorem}, our structure theorem for the Bergman projection. To state our results, we need the following notation.
Let $\I_q = \{ J = (j_1,\dots,j_q) \in \N^q : 1 \leq j_1 < \cdots < j_q \leq n\}$ be the set of increasing $q$-tuples and let 
\[
\widetilde{d\z_j} = d\z_1 \wedge \cdots \wedge \widehat{d\z_j} \wedge \cdots \wedge d\z_n
\]
where $\widehat{d\z_j}$ represents the omission of $d\z_j$ from the wedge product. We will also use the $[\hat I]$ to denote the $(n-|I|)$-tuple $\{1,\dots,n\}\setminus I$.

\begin{thm}\label{thm:structure theorem}
Let $\Om\subset\C^n$ be a domain and $1 \leq q \leq n-1$. Then: 
\begin{enumerate}
\item There exists an integral kernel $B_q(z,w)$ so that the Bergman projection $B_q:L^2_{0,q}(\Om) \to L^2_{0,q}(\Om)\cap\ker \dbar_q$ is given by
\[
B_q f(z) = \int_\Om f(w) \wedge B_q(z,w)
\]
for any $f \in L^2_{0,q}(\Om)$;
\item  Moreover, there exist bounded operators $B_{J'J}:L^2(\Om)\to L^2(\Om)$ so that if $f = \sum_{J\in\I_q} f_J\,d\z^J$, then
\[
B_q f(z) = \sum_{J,J'\in\I_q} B_{J'J} f_J(z)\, d\z^{J'};
\]
\item Given any orthonormal basis $\{\phi_j\} \subset L^2_{0,q}(\Om)\cap \ker\dbar$,
\[
B_q(z,w) = \sum_{j=1}^\infty \star\phi_j(w) \wedge \phi_j(z)
\]
where the sum converges in $L^2_{(0,q),(n,n-q)}(\Om\times\Om)$.
\end{enumerate}
\end{thm}

We have additional information about the operators $B_{J'J}$ in the case that $q=n-1$.
\begin{thm}\label{thm:n-1 case}
Let $\Om\subset\C^n$ be a domain and $G(z,w)$ be the Green's function for the Laplacian $\triangle$. Then 
\begin{enumerate}
\item
\[
B_{n-1}f(z) = f(z) - \int_\Om f(w) \wedge  \vartheta_{n-1,z}\p^*_{n-1,w}N_n(z,w);
\]
\item 
\begin{equation}\label{eqn:B tilde j tilde k general}
B_{[\hat k] [\hat j]}(z,w) = \delta_{jk} \delta_z(w)+ (-1)^{n+j+k-1}4  \frac{\p^2 G(z,w)}{\p z_k\p\w_j}
\end{equation}
where $\delta_{jk}$ is the Kronecker $\delta$ and $\delta_z(w)$ is the Dirac $\delta$.
\item 
In the case that $\Om = B(0,1)$ is the unit ball then
\begin{align*}
&B_{[\hat k] [\hat j]}(z,w) = \delta_{jk}\delta_z(w)+ (-1)^{n+j+k-1} \frac{(n-1)!}{\pi^n} \bigg[  \frac{\delta_{jk}}{|z-w|^{2n}} - n\frac{(z_k-w_k)(\z_j-\w_j)}{|z-w|^{2n+2}} \\
&- \frac{\delta_{jk} - \z_j w_k }{(|z-w|^2 + (1-|w|^2)(1-|z|^2))^n} + n \frac{((z_k-w_k)+w_k(1-|z|^2))((\z_j-\w_j)-\z_j(1-|w|^2))}{(|z-w|^2 + (1-|w|^2)(1-|z|^2))^{n+1}}\bigg]
\end{align*}
\end{enumerate}
\end{thm}

Our final result is the failure of the boundedness of pointwise evaluation in $L^2_{0,q}(\Om)\cap\ker \dbar_q$, $q\geq 1$. This result stands in stark contract to $B_0$, and, in fact, boundedness pointwise evaluation
in $L^2(\Om)$
is a critical fact for $B_0$ and (more generally) one of the defining assumptions in
the expansive theory of reproducing kernel Hilbert spaces, see, e.g., \cite{ChTo04}.
To observe the first instance of the boundedness of pointwise evaluation in the theory of the Bergman project, we simply need to recall the standard construction for $B_0$.
This construction works equally well for reproducing kernels in reproducing kernel Hilbert spaces.
Suppose that the evaluation functional $e_z(\vp) = \vp(z)$ was a bounded, linear functional, 
i.e., $|e_z(\vp)| \leq C \|\vp\|_{L^2_{0,q}(\Om)}$ for some constant $C$ that may depend on $z$ but not on $\vp$. 
This would mean for any $f \in \ker \dbar_q \cap L^2_{0,q}(\Om)$, $|f(z)| \leq C \|f\|_{L^2_{0,q}(\Om)}$ where $C = C(z)$ does not depend on $f$. This is critical for the following reason: for any $\{a_j\} \in \ell^2$, 
$f(z) = \sum_{j=1}^\infty a_j \vp_j(z) \in \ker \dbar_q \cap L^2_{0,q}(\Om)$, with the consequence that
\begin{align*}
|K(z,z)| = \sum_{j=1}^\infty |\vp_j(z)|^2\, dV(z) =\Big( \sup_{\atopp{\{a\}\in \ell^2}{\|a\|_{\ell^2}=1}} \Big| \sum_{j=1}^\infty a_j \vp_j(z)  \Big| \Big)^2\, dV(z) =  \sup_{\atopp{f\in\ker\dbar}{ \|f\|_{L^2}=1}} | f(z) |^2\, dV(z)
\end{align*}
Consequently, boundedness on the diagonal implies finiteness of $\sup_{\atopp{f\in\ker\dbar}{ \|f\|_{L^2}=1}}  | f(z) |$. From Theorem \ref{thm:n-1 case}, it is immediate that $B_{n-1}(z,w)$ blows
up as $w \to z$.

\begin{thm}\label{thm:pointwise evaluation fails}
Let $\Om\subset\C^n$ be a domain. If $1 \leq q \leq n$, then pointwise evaluation is not a bounded, linear functional on $L^2_{0,q}(\Om)\cap \ker\dbar_q$.
\end{thm}

\begin{proof} Since forms are not functions, we consider pointwise evaluation to be the pointwise evaluation functionals $\vp \mapsto \vp_J$ for each $J\in \I_q$. 
Without loss of generality, we may suppose that $0 \in \Om$. 
Let $q \geq 1$, $J\in\I_q$, and $\psi\in (C^\infty_c)_{0,q-1}(\Om)$ so that $(\dbar\psi(0))_J \neq 0$.
Set $\vp(z) = \frac{\dbar \psi(z)}{\|(\dbar \psi)_J\|_{L^2(\Om)}}$. Then $\vp_\ep(z) = \ep^{-\frac n2} \vp(z/\ep) \in (C^\infty_c)_{0,q}(\Om) \cap \ker\dbar_q$
since $\dbar^2=0$. Moreover, our normalization ensure $\|(\vp_\ep)_J\|_{L^2(\Om)}=1$ for all $\ep>0$ but $|(\vp_\ep)_J(z)| \to \infty$ as $\ep\to 0$. 
\end{proof}

\begin{rem} It is very unlikely that the Bergman kernel $B_{p,q}(z,w)$ satisfies 
a nice transformation formula under biholomorphisms unless $p,q\in\{0,n\}$. The transformation law for $B_0$ essentially follows from the pullback relationship
$F^* \dbar = \dbar F^*$ and the fact that $J_\R F = |J_\C F|^2$ where $J_\R F$ is the determinant of the real Jacobian and $J_\C F$ is the determinant of the complex Jacobian. In general, while the pullback interacts nicely with
$\dbar$, it behaves poorly with respect to $L^2$-inner products. In particular, if $F:\Omega_1 \to \Omega_2$ is a biholomorphism and $\phi,\psi\in L^2_{p,q}(\Om_2)$, then
\begin{align*}
&\big(F^*\phi,F^*\psi\big) = \int_{\Omega_1} F^* \phi(w) \wedge \star \big(F^*\psi(w)\big) \\
&= \sum_{\atopp{I,I',K\in\I_p}{J,J',L'\in\I_q}} \int_{\Omega_1} \Big(\phi_{IJ}\circ F(w)\Big)\Big(\overline{\psi_{I',J'}\circ F(w)}\Big) \bigg|\frac{\p F^I}{\p w^K}\bigg| \bigg|\overline{\frac{\p F^J}{\p w^L}}\bigg|
\bigg|\overline{\frac{\p F^{[\hat I]}}{\p w^{[\hat K]}}}\bigg| \bigg| \frac{\p F^{[\hat J]}}{\p w^{[\hat L]}}\bigg|\, dV(w)
\end{align*}
where $\frac{\p F^I}{\p w^K}$ is the $p\times p$ minor of the complex Jacobian of the mapping $F = (F_1,\dots,F_n)$ given by
\[
\frac{\p F^I}{\p w^K} = \Big( \frac{\p F_{I_j}}{\p w_{K_k}} \Big)_{j,k=1}^p
\]
where $I = (I_1,\dots,I_p)$ and $K = (K_1,\dots,K_p)$ and similarly for the other terms.
The complicated product of determinants only simplifies dramatically in the cases $p,q\in\{0,n\}$ to $J_\R F$ and a change of variables may proceed as in the $B_0$ case.
\end{rem}

\subsection{Existence of the Bergman kernel and the proof of Theorem \ref{thm:structure theorem}} 
We know that the Bergman projection is a bounded, linear operator. We now show that $B_q$ is an integral operator and that the Bergman kernel exists.
Given $f\in L^2_{0,q}(\Om)$, we can write
\[
f = \sum_{J\in\I_q} f_J\, d\z^J
\]
The Bergman projection is a linear operator so that 
\[
B_q(f_J\, d\z^J) = \sum_{J'\in\I_q} (B_q f_J\, d\z^J)_{J'} \, d\z^{J'}.
\]
Define
\[
B_{J'J}' : L^2_{0,q}(\Om) \to L^2_{0,q}(\Om)
\]
by the mapping
\begin{equation}\label{eqn:B_J'J}
B_{J' J}' f = (B_q f_J\, d\z^J)_{J'} \, d\z^{J'}. 
\end{equation}
It is easy to see that the operator norm $\|B_{J'J}\|_{L^2 \to L^2} \leq 1$ and
\[
B_q f = \sum_{J,J'\in\I_q} B_{J'J}'\big( f_J\, d\z^J\big) 
\]

For each operator $B_{J'J}'$, define an auxiliary operator $B_{J'J}:L^2(\Om)\to L^2(\Om)$ that satisfies
\[
B_{J'J} f_J = \big(B'_{J'J}f_J\big)_{J'}
\]
Essentially, $B_{J'J}$ the operator that takes the coefficient of $f$ on $d\z^J$ and maps it to the $d\z^{J'}$ coefficient of $B_{J'J}' f$.

Recall the Schwartz Kernel Theorem \cite[Theorem 5.2.1]{Hor90a}. We state a version of it for our particular setup. Every function $K \in C(\Om\times\Om)$ defines an integral operator
$\opK$ from $C_c(\Om)$ to $C(\Om)$ by the formula
\[
\opK \vp(x_1)  = \int_\Om K(x_1,x_2) \vp(x_2)\, dV(x_2),\quad \vp\in C_c(\Om),\ x_1\in\Om.
\] 
The Schwartz Kernel Theorem extends this definition to arbitrary distributions $K$ if $\vp$ is restricted to $C^\infty_c(\Om)$ and $\opK\vp$ is allowed to be a distribution. The first observation is that if
$K\in C(\Om\times\Om)$, then
\[
\la \opK \vp,\psi\ra = K(\psi\otimes\vp) = \int_\Om\int_\Om K(x_1,x_2) \vp(x_2) \psi(x_1)\, dV(x_2)\, dV(x_1).
\]

\begin{thm}[Schwartz Kernel Theorem] 
Every $K \in \mathcal D'(\Om\times\Om)$ defines according to
\begin{equation}\label{eqn:SKT distribution equality}
\la \opK \vp,\psi\ra = K(\psi\otimes\vp); \quad \psi,\vp\in C^\infty_c(\Om)
\end{equation}
a linear map $\opK$ from $C^\infty_c(\Om)$ to $\opD'(\Om)$ which is continuous in the sense that $\opK \vp_j\to 0$ in $\opD'(\Om)$ if $\vp_j\to 0$ in $C^\infty_c(\Om)$. Conversely, to every such linear map
$\opK$ there is one and only one distribution $K$ such that \eqref{eqn:SKT distribution equality} is valid. One calls $K$ the kernel of $\opK$.
\end{thm}
Since the maps $B_{J'J} :L^2(\Om)\to L^2(\Om)$ boundedly, they certainly map from $C^\infty_c(\Om) \to \opD'(\Om)$. Consequently, the Schwartz Kernel Theorem applies to each $B_{J'J}$. As a result, the Bergman kernel 
on $(0,q)$-forms exists as a distributional kernel, and we can write (for $f,g\in \opD_{0,q}(\Om)$)
\[
(B_qf, g) = \int_\Om \int_\Om f(w) \wedge B_q(z,w) \wedge * g(z)\, dV(w)\, dV(z) = K_q(f \otimes g)
\]
where the integral is understood in the distributional sense.

We now turn to establishing greater regularity for $B_q(z,w)$. Let $\{\phi_j\}$ be an orthonormal basis of $L^2_{0,q}(\Om) \cap \ker\dbar_q$,
\[
K_N(z,w) = \sum_{j=1}^N *\phi_j(w) \wedge \phi_j(z),
\]
and $\opK_N$ as the operator with kernel $K_N$. We will show that
\[
K_N(z,w) \to B_q(z,w) \text{ in }L^2_{(0,q),(n,n-q)}(\Om).
\]
Since $\opK_N f \to Bf$ in $L^2_{0,q}(\Om)$ and $\{\phi_j\}$ are orthogonal, there exists $N'>0$ so that if $M \geq N \geq N'$, then
\[
\Big| K_M(f\otimes g) - K_N(f \otimes g) \Big| = \Big| \int_{\Om\times\Om} \sum_{j=N+1}^M f(w) \wedge *\phi_j(w)\wedge\phi_j(z) \wedge *g(z) \Big| < \ep \|f\|_{L^2(\Om)}\|g\|_{L^2(\Om)}.
\]
Consequently, the sequence of operators $\{\opK_N\}$ with distributional kernels $\{K_N\}$ forms a Cauchy sequence acting on $L^2_{0,q}(\Om)\times L^2_{0,q}(\Om)$ and therefore converges to an operator $B'$ acting on
$L^2_{0,q}(\Om)\times L^2_{0,q}(\Om)$ and with distributional kernel $K(z,w)$. Moreover, since $K_N(z,w)$ forms a Cauchy sequence in $L^2_{0,q}(\Om)\otimes L^2_{n,n-q}(\Om)$, it follows that 
$B_q(z,w) \in L^2_{0,q}(\Om)\otimes L^2_{n,n-q}(\Om) \subset L^2_{(0,q),(n,n-q)}(\Om\times\Om)$. That this sum is independent of the basis is a standard Hilbert space fact. This concludes the proof of Theorem \ref{thm:structure theorem}.

%
%
\section{The Bergman projection $B_{n-1}$ and the proof of Theorem \ref{thm:n-1 case}, parts (1) and (2)}\label{sec:B n-1}

Recall that the boundary condition for a form $u = \sum_{J\in\I_q} u_J\, d\z^J \in L^2_{0,q}(\Om)$ to be an element of $\Dom(\dbars)$ is that
\[
\sum_{j=1}^n u_{jK} \frac{\p\rho}{\p z_j} =0 \text{ in $\bd\Om$ for all } K\in\I_{q-1}
\]
where
\[
u_{jK} = \sum_{J\in\I_q} \ep^{jK}_J u_J.
\]
If $q=n$ the boundary requirement is exactly that $u_{\{1,\dots,n\}} \frac{\p\rho}{\p z_j}=0$ for all $j=1,\dots,n$, i.e., $u=0$ on $\bd\Om$. This is the Dirichlet boundary condition and the $\dbar$-Neumann problem reduces to the standard
Dirichlet problem for the Laplacian. We normalize the Laplacian $\triangle$ so that $\triangle = -4\sum_{j=1}^n \frac{\p^2}{\p z_j\p\z_j}$.
Consequently, if $G(z,w)$ is the Green's function for the Laplacian
on $\Om$, then the $\dbar$-Neumann operator on the top degree is
\[
N_n(z,w) = 4 G(z,w)\, dw \wedge\, d\z
\]
with the notation $dw = dw_1\wedge \cdots \wedge dw_n$ and $d\z = d\z_1 \wedge \cdots \wedge d\z_n$. The integral operator $N_n$ applied to a $(0,n)$-form $F = f \, d\z$ is then
\[
N_n F(z) = \int_\Om F(w) \wedge N_n(z,w) \wedge d\z = 4\Big[\int_\Om f(w) G(z,w)\, dV(w) \Big]d\z.
\]
Thus we have an explicit integral kernel for $N_n$ for every case for which there is an explicit formula for $G(z,w)$.

Recall Kohn's formula for the Bergman projection:
\[
B_q = I - \dbars_q N_{q+1} \dbar_q = I - \vartheta_q N_{q+1} \dbar_q.
\]
We now compute $B_{n-1}$ and recall that $G(x,y)=0$ whenever either $x\in\bd\Om$ or $y\in\bd\Om$. Suppose $f\in L^2_{0,n-1}(\Om)$. Then
\begin{align*}
B_{n-1}f(z) &= f(z) - \vartheta_{n-1,z}  \int_\Om \dbar_{n-1,w}f(w) \wedge N_n(z,w)  \\
&= f(z) - \vartheta_{n-1,z} \int_\Om f(w) \wedge \p^*_{n-1,w} N_n(z,w).
\end{align*}
We would like to bring the operator $\vartheta_{n-1,z}$ inside the integral but this requires care because the Newtonian potential on $\C^n$ is
\[
\Phi(z) = \frac{(n-2)!}{4\pi^n} \frac{1}{|z|^{2n-2}}
\]
and two derivatives means that the kernel would blow up like a singular integral. 
In point of fact, this will not cause a problem because derivatives of two derivatives of $\Phi(z)$ generate a Calder\'on-Zygmund singular integral. But care certainly
must be taken! In particular, the Green's function $G(z,w)$ is built from the Newtonian potential and a harmonic function. 
Therefore, the singularity of $\frac{\p^2 G(z,w)}{\p z_j \p \w_k}$ can only come from the $\frac{\p^2}{\p z_j \p \w_k} \frac{1}{|w-z|^{2n-2}}$ which we now compute.
\[
\frac{\p^2}{\p z_j \p \w_k} \frac{1}{|w-z|^{2n-2}}
= (n-1)\frac{\delta_{jk}}{|w-z|^{2n}} - n(n-1) \frac{(w_k-z_k)\overline{(w_j-z_j)}}{|w-z|^{2(n+1)}}
\]
The case $j\neq k$ yields the kernel $ \frac{(w_k-z_k)\overline{(w_j-z_j)}}{|w-z|^{2(n+1)}}$ which is a classic Calder\'on-Zygmund convolution kernel -- homogeneous of degree $-2n$ and integrates to $0$ over 
any sphere centered around the origin.
The case $j=k$ is only slightly more complicated. Observe that if $\sigma_{2n-1}$ is the surface area of the unit sphere in $\C^n$, then by symmetry
\begin{align*}
\int_{\bd B(0,1)} \frac{n-1}{|z|^{2n} } &- \frac{n(n-1)|z_j|^2}{|z|^{2(n+1)}}\, d\sigma(z)
= (n-1)\sigma_{2n-1} - n(n-1) \int_{\bd B(0,1)} |z_j|^2\, d\sigma(z) \\
&= (n-1)\sigma_{2n-1} - n(n-1) \int_{\bd B(0,1)} \frac 1n \sum_{k=1}^n |z_j|^2\, d\sigma(z) = 0.
\end{align*}
By homogeneity, the integral is $0$ around any sphere, thus we can write
\[
B_{n-1}f(z) = f(z) -  \int_\Om f(w) \wedge \vartheta_{n-1,z}\p^*_{n-1,w} N_n(z,w)
\]
where the integral is taken in the sense of (tempered) distributions. A version of this formula (written directly in terms of the Green's function) appears in \cite[Theorem 15.3]{Bel92} for domains in $\C$ and the Bergman projection $B_0$.
Breaking down $B_{n-1}$ into its constituent parts, we compute
\begin{align*}
-\vartheta_{n-1,z}\p^*_{n-1,w} N_n(z,w) &= -\vartheta_{n-1,z}\p^*_{n-1,w} \Big\{ 4 G(z,w)\, dw \wedge d\z\Big\} \\
&= 4 \vartheta_{n-1,z} \Big\{ \sum_{k=1}^n (-1)^{k-1} \frac{\p G(z,w)}{\p\w_k} \widetilde{dw_k} \wedge d\z\Big\} \\
&= (-1)^{n-1}4 \sum_{j,k=1}^n (-1)^{j+k} \frac{\p^2 G(z,w)}{\p z_j\p\w_k}\, \widetilde{dw_k} \wedge \widetilde{d\z_j}.
\end{align*}
from which \eqref{eqn:B tilde j tilde k general} follows.

\subsection{The proof of Theorem \ref{thm:n-1 case}, parts (3) and (4)}\label{sec:examples}

We now restrict ourselves to the case $\Om$ is the unit ball on which the Green's function 
\[
G(z,w) = \Phi(z-w) - \Phi(|w|(z-\tilde w)) = \Phi(w-z) - \Phi(|z|(w-\tilde z))
\]
where $\tilde w = \frac{w}{|w|^2}$ is the reflection of $w$ across the unit sphere. 
Since 
\[
|z|^2|w-\tilde z|^2 - |w-z|^2 = |z|^2|w|^2 + 1 -|w|^2-|z|^2 = (1-|w|^2)(1-|z|^2),
\]
it follows that
\begin{equation}\label{eqn:Green on ball pretty}
G(z,w) = \frac{(n-2)!}{4\pi^n}\bigg( \frac{1}{|z-w|^{2n-2}} - \frac{1}{(|z-w|^2 + (1-|w|^2)(1-|z|^2))^{n-1}}\bigg)
\end{equation}

In this case, note that
\[
\frac{\p G(z,w)}{\p \w_k} =   \frac{(n-1)!}{4\pi^n}\bigg( \frac{z_k-w_k}{|z-w|^{2n}} - \frac{z_k-w_k + w_k(1-|z|^2)}{(|z-w|^2 + (1-|w|^2)(1-|z|^2))^n}\bigg)
\]
and so $\frac{\p G(z,w)}{\p \w_k} \equiv 0$ whenever $w\in B(0,1)$ and $z \in \bd B(0,1)$ (reflecting the fact that $N_n \dbar_{n-1}\in\Dom(\dbars_{n-1})$). Also,
\begin{align*}
&\frac{\p^2 G(z,w)}{\p z_j \p\w_k}
= \frac{(n-1)!}{4\pi^n} \bigg[  \frac{\delta_{jk}}{|z-w|^{2n}} - n\frac{(z_k-w_k)(\z_j-\w_j)}{|z-w|^{2n+2}} \\
&- \frac{\delta_{jk} - \z_j w_k }{(|z-w|^2 + (1-|w|^2)(1-|z|^2))^n} + n \frac{((z_k-w_k)+w_k(1-|z|^2))((\z_j-\w_j)-\z_j(1-|w|^2))}{(|z-w|^2 + (1-|w|^2)(1-|z|^2))^{n+1}}\bigg]
\end{align*}
from which part (3) of Theorem \ref{thm:n-1 case} follows.

%
%
\section{Control geometry and the unit ball in $\C^2$}\label{subsec:unit ball in C^2}
Observe that if $z \to \bd B(0,1)$, then
\[
B_{[\hat k] [\hat j]}(z,w) = \delta_{jk}\delta_z(w) - (-1)^{j+k} \frac{1}{\pi^2} \bigg[  \frac{\z_jw_k}{|z-w|^4} - 2 \frac{\z_j(z_k-w_k)(1-|w|^2)}{|z-w|^6}\bigg].
\]
as $z \to\bd B(0,1)$. Let $a_{jk}(z,w) = \frac{\z_jw_k}{|z-w|^4} - 2 \frac{\z_j(z_k-w_k)(1-|w|^2)}{|z-w|^6}$.

A defining function for $B(0,1)$ is $r(z) = |z|^2-1$. Consequently, the $(1,0)$ complex tangential vector field is $L = \z_2 \frac{\p}{\p z_1} - \z_1 \frac{\p}{\p z_2}$ and the complex normal
is given by $S = 2z_1 \frac{\p}{\p z_1} + 2z_2 \frac{\p}{\p z_2}$. Observe that $[L,\bar L] = - \Imm S$. If $z=(0,1)$ and $w = (w_1,1-h)$, then $a_{1k}((0,1),w)=0$ and
\[
a_{22}\big((0,1),(w_1,1-h)\big) = \frac{1+h}{(|w_1|^2+|h|^2)^2} - \frac{ 4h\Rre h}{(|w_1|^2 + |h|^2)^3}
\]
and
\[
a_{21}\big((0,1),(w_1,1-h)\big) = - \frac{w_1}{(|w_1|^2+|h|^2)^2} + \frac{4w_1 \Rre h}{(|w_1|^2+|h|^2)^3}.
\]
while the Bergman kernel 
\[
B_0\big((0,1),(w_1,1-h)\big) = -\frac{2}{\pi^2 \bar h^3}.
\]
For the proper size estimate comparisons with $B_0(z,w)$, we recall the control metric from \cite{NaStWa85} and the Bergman kernel estimates of \cite{NaRoStWa89,McN89}. At $(0,1)$, note that $L = \frac{\p}{\p z_1}$ and
$S = 2\frac{\p}{\p z_2}$ which means that the distance from $(0,1)$ in the $w_1$-direction is weighted by order 1 and in the $w_2$-direction by order 2. In other words, $d((0,1),(w_1,1-h)) \approx |w_1| + |h|^{1/2}$. It is clear that
$a_{2k}(z,w)$ observes different scaling and size estimates that $B_0(z,w)$ as $|w_1|$ appears with the same weighting as $|h|$. Once again, $B_1$ behaves quite differently than $B_0$!
%
%
\section{Conclusion}\label{sec:conclusion}
This paper checks the functional analysis to show that the Bergman projection has a well-defined integral kernel and that $B_{n-1}(z,w)$ is quite computable from the Green's function $G(z,w)$. Of course, computing the 
Green's function for domains of interest in several complex variables (and domains in general) is a complicated task. We will return to this topic in a future paper, in particular for $n=2$ case, as we can say much more there.

\bibliographystyle{alpha}
\bibliography{mybib}

\begin{thebibliography}{NRSW89}

\bibitem[Bel92]{Bel92}
S.~Bell.
\newblock {\em The {C}auchy transform, potential theory, and conformal
  mapping}.
\newblock Studies in Advanced Mathematics. CRC Press, Boca Raton, FL, 1992.

\bibitem[BS90]{BoSt90}
H.\ Boas and E.\ Straube.
\newblock Equivalence of regularity for the {B}ergman projection and the
  {$\overline \partial$}-{N}eumann operator.
\newblock {\em Manuscripta Math.}, 67(1):25--33, 1990.

\bibitem[BTA04]{ChTo04}
A.~Berlinet, , and C.~Thomas-Agnan.
\newblock {\em Reproducing kernel {H}ilbert spaces in probability and
  statistics}.
\newblock Kluwer Academic Publishers, Boston, MA, 2004.

\bibitem[Cat83]{Cat83}
D.\ Catlin.
\newblock Necessary conditions for subellipticity of the
  $\overline\partial$-{N}eumann problem.
\newblock {\em Ann.\ of Math.}, 117:147--171, 1983.

\bibitem[Cat87]{Cat87}
D.\ Catlin.
\newblock Subelliptic estimates for the $\overline\partial$-{N}eumann problem
  on pseudoconvex domains.
\newblock {\em Ann.\ of Math.}, 126:131--191, 1987.

\bibitem[CD06]{ChDu06}
P.~Charpentier and Y.~Dupain.
\newblock Estimates for the {B}ergman and {S}zeg{\"o} projections for
  pseudoconvex domains of finite type with locally diagonalizable {L}evi form.
\newblock {\em Publ.~Mat.}, 50(2):413--446, 2006.

\bibitem[CS01]{ChSh01}
S.-C.\ Chen and M.-C.\ Shaw.
\newblock {\em Partial Differential Equations in Several Complex Variables},
  volume~19 of {\em Studies in Advanced Mathematics}.
\newblock American Mathematical Society, 2001.

\bibitem[D'A78]{DAn78}
J.~P. D'Angelo.
\newblock A note on the {B}ergman kernel.
\newblock {\em Duke Math.\ J.}, 45(2):259--265, 1978.

\bibitem[D'A94]{DAn94}
J.~P. D'Angelo.
\newblock An explicit computation of the {B}ergman kernel function.
\newblock {\em J.~Geom.~Anal.}, 4(1):23--34, 1994.

\bibitem[Fef74]{Fef74}
C.~Fefferman.
\newblock The {B}ergman kernel and biholomorphic mappings of pseudoconvex
  domains.
\newblock {\em Invent. Math.}, 26:1--65, 1974.

\bibitem[FK72]{FoKo72}
G.~B.\ Folland and J.~J.\ Kohn.
\newblock {\em The {Neumann} problem for the {Cauchy}-{Riemann} Complex},
  volume~75 of {\em Ann.\ of Math.\ Stud.}
\newblock Princeton University Press, Princeton, New Jersey, 1972.

\bibitem[H{\"o}r90]{Hor90a}
L.\ H{\"o}rmander.
\newblock {\em The analysis of linear partial differential operators. {I}}.
\newblock Springer Study Edition. Springer-Verlag, Berlin, {S}econd edition,
  1990.

\bibitem[HR15]{HaRa15}
P.~Harrington and A.~Raich.
\newblock Closed range for $\bar\partial$ and $\bar\partial_b$ on bounded
  hypersurfaces in {S}tein manifolds.
\newblock {\em Ann. Inst. Fourier (Grenoble)}, 65(4):1711--1754, 2015.

\bibitem[KN65]{KoNi65}
J.J. Kohn and L.\ Nirenberg.
\newblock Non-coercive boundary value problems.
\newblock {\em Comm.\ Pure Appl.\ Math.}, 18:443--492, 1965.

\bibitem[KR]{KhRa18}
T.V. Khanh and A.~Raich.
\newblock Local regularity of the {B}ergman projection on a class of
  pseudoconvex domains of finite type.
\newblock {\em submitted}.
\newblock arXiv:1406.6532.

\bibitem[Kra01]{Kra01}
S.\ Krantz.
\newblock {\em Function theory of several complex variables}.
\newblock AMS Chelsea Publishing, Providence, RI, {R}eprint of the 1992 edition
  edition, 2001.

\bibitem[McN89]{McN89}
J.D.\ McNeal.
\newblock Boundary behavior of the {B}ergman kernel function in ${{\mathbb
  C}}^2$.
\newblock {\em Duke Math.\ J.}, 58:499--512, 1989.

\bibitem[McN94]{McN94}
J.D.\ McNeal.
\newblock Estimates on the {B}ergman kernels of convex domains.
\newblock {\em Adv.\ Math.}, 109:108--139, 1994.

\bibitem[MS94]{McSt94}
J.D.\ McNeal and E.M.\ Stein.
\newblock Mapping properties of the {B}ergman projection on convex domains of
  finite type.
\newblock {\em Duke Math.\ J.}, 73(1):177--199, 1994.

\bibitem[NRSW89]{NaRoStWa89}
A.\ Nagel, J.-P.\ Rosay, E.M.\ Stein, and S.\ Wainger.
\newblock Estimates for the {B}ergman and {S}zeg{\"o} kernels in ${{\mathbb
  C}}^2$.
\newblock {\em Ann.\ of Math.}, 129:113--149, 1989.

\bibitem[NS06]{NaSt06}
A.\ Nagel and E.M.\ Stein.
\newblock The $\dbar_b$-complex on decoupled domains in ${{\mathbb C}}^n$, $n
  \geq 3$.
\newblock {\em Ann.\ of Math.}, 164:649--713, 2006.

\bibitem[NSW85]{NaStWa85}
A.\ Nagel, E.M.\ Stein, and S.\ Wainger.
\newblock Balls and metrics defined by vector fields {I}: Basic properties.
\newblock {\em Acta Math.}, 155:103--147, 1985.

\bibitem[PS77]{PhSt77}
D.H.\ Phong and E.M.\ Stein.
\newblock Estimates for the {B}ergman and {S}zeg{\"o} projections on strongly
  pseudo-convex domains.
\newblock {\em Duke Math.\ J.}, 44(3):695--704, 1977.

\end{thebibliography}

\end{document}